\documentclass[11pt]{article}

\usepackage[a4paper,margin=1in]{geometry}

\usepackage{graphicx}%
\usepackage{multirow}%
\usepackage{amsmath,amssymb,amsfonts}%
\usepackage{amsthm}%
\usepackage{mathrsfs}%
\usepackage[title]{appendix}%
\usepackage{xcolor}%
\usepackage{textcomp}%
\usepackage{enumerate}
\usepackage{epsfig}
\usepackage{verbatim}
\usepackage{bm} 
\usepackage{booktabs}
\usepackage{subcaption}
\usepackage{caption}

\providecommand{\e}{\varepsilon}

\providecommand{\R}{\ensuremath{\mathbb{R}}}
\providecommand{\N}{\ensuremath{\mathbb{N}}}

\newcommand{\realiz}{\mathrm{R}}

\newcommand{\depth}{L}
\newcommand{\size}{M}

\newcommand{\ceil}[1]{\lceil #1 \rceil}

\theoremstyle{plain}
\newtheorem{theorem}{Theorem}[section]
\newtheorem{proposition}[theorem]{Proposition}

\theoremstyle{definition}

\newtheorem{definition}[theorem]{Definition}

\theoremstyle{remark}
\newtheorem{remark}[theorem]{Remark}

\numberwithin{equation}{section}

\title{A simple shallow neural network for emulating the solution to singularly perturbed problems}

\author{C.~Xenophontos\thanks{Corresponding author: \texttt{xenophontos.christos@ucy.ac.cy}} \and
A.~Raina\thanks{Email: \texttt{raina.aayushman@ucy.ac.cy}}}
\date{Department of Mathematics and Statistics, University of Cyprus,\\
P.O.~Box 20537, Nicosia 1678, Cyprus}

\begin{document}

\maketitle

\begin{abstract}
We consider (feed-forward) Neural Networks (NNs) for the emulation of the solution to singularly 
perturbed second order boundary value problems, with two small parameters. We describe 
a shallow NN which exploits available asymptotic expansions for the solution. These additive decompositions into smooth and layer components, allow for derivative estimates which are explicit in the order of differentiation as well as the singular perturbation parameter(s) \cite{melenk, Irene, SX}. Utilizing such 
decompositions, we propose a simple NN for emulating the solution to such  problems
using the $\tanh$ activation function together with different training objectives, such as residual or energy minimization. The key idea is to augment the approximation space with suitable exponential functions, similar to enriched spaces in finite element methods, e.g.~\cite{Kellogg}. Numerical examples in one and two dimensions, including a smooth non-tensor-product domain, illustrate the resulting parameter-robust behavior over the tested perturbation ranges.
\end{abstract}



\section{Introduction}
\label{sec:intro}

Singularly perturbed problems arise naturally in a wide range of scientific and engineering applications, including fluid and solid mechanics, electromagnetics, and many others. In such settings, the governing partial differential equations involve a small parameter $\e\in(0,1]$ in the physically relevant regime, leading to singular perturbations.

A characteristic feature of singularly perturbed problems (SPPs) is that their solutions admit an additive decomposition into a regular component and one or more layer components. When the problem data are analytic, the regular component is analytic as well. Although the regular part, denoted by $u_\e^{S}$, may itself depend on the perturbation parameter $\e$, its derivatives satisfy estimates that remain uniform with respect to $\e$. We refer the reader to \cite{TemamSgPert} and the references therein for further details.
In contrast, the layer component $u_\e^{BL}$ does not possess uniform smoothness as $\e$ varies. Instead, its derivatives become increasingly large as $\e$ decreases, with the $k$-th derivative typically exhibiting growth of order $O(\e^{-k})$.

When approximating the solution $u_\e$ numerically, for example by the Finite Element Method (FEM), it is generally necessary to employ a \emph{layer-adapted} mesh (see, e.g., \cite{L}), whose construction depends explicitly on the singular perturbation parameter $\e$. Well-known examples for fixed-order FEMs are the Shishkin \cite{Shishkin2} and Bakhvalov \cite{B} meshes. For high-order $p$-FEM discretizations, a widely used alternative is the \emph{Spectral Boundary Layer} (SBL) mesh introduced in \cite{MXO} (see also \cite{SS}). Although the SBL mesh consists of only a fixed number of elements, it is specifically designed to resolve the individual components of the asymptotic solution decomposition within the regions where they are significant. As a result, it yields approximation errors that are uniform with respect to $\e$ and converge exponentially as the polynomial degree $p$ increases \cite{melenk}.

A similar framework applies, with suitable modifications, to problems involving \emph{two} singular perturbation parameters associated with the two highest-order derivative terms. In this case, the asymptotic expansion again separates into smooth and layer components, although the structure of the layers is now determined by the interaction between the two perturbation parameters (see, e.g., \cite{SX}). The layer-adapted meshes described above can be generalized to this setting, leading to approximation methods that remain robust with respect to both perturbation parameters (see, e.g., \cite{Irene}).

Recently, Neural Networks (NNs) have emerged as a powerful tool for the numerical approximation of partial differential equations; see, e.g., \cite{Huang} and the references therein. SPPs, however, remain particularly challenging because their solutions could contain boundary and interior layers whose widths may vary by several orders of magnitude as the perturbation parameter(s) change. Standard NN architectures must therefore learn simultaneously both the smooth component of the solution and the rapidly varying layer components, a task that often leads to poor robustness when the perturbation parameter becomes very small.

The literature on NN approximations for SPPs has grown rapidly in recent years; see, for example, \cite{AD, ACD}, \cite{BEFFM}--\cite{CGGY}, \cite{WY}--\cite{Xu}, \cite{OSX, OPS, XS}. Existing approaches range from approximation theoretic expressivity results to practical algorithms employing a variety of architectures and training strategies. Nevertheless, to the best of our knowledge, no single architecture has been shown to provide uniformly accurate approximations across all perturbation regimes using the same underlying network design.

The main idea of the present work is that the asymptotic structure of the solution should determine the approximation space used by the neural method. Rather than expecting the optimizer to discover the boundary layers, we explicitly incorporate the known layer behavior through suitable exponential enrichments. The shallow $\tanh$ component is therefore relieved of the need to represent the dominant boundary-layer scales on its own, while the enriched features provide functions with the correct localization and parameter-dependent widths. This leads to a remarkably simple architecture, which we call the \emph{Boundary Layer Neural Network} (BL-NN). In approximation-theoretic terms, the construction is an enriched space, analogous in spirit to enriched finite element spaces; see, e.g., \cite{Kellogg}. Related structure-informed NN approaches for singular perturbations include \cite{ACD, CGGY, GHJL}.

The distinction we emphasize is that the principal modification is made to the \emph{approximation space} rather than to a particular loss functional or optimization algorithm. In one dimension the implemented BL-NN consists of a one-hidden-layer $\tanh$ network augmented by two explicit exponential layer features. Thus it is more precise to regard the method as an \emph{exponentially enriched shallow $\tanh$ network}, rather than as a pure $\tanh$ network. Its design is directly motivated by the asymptotic decomposition of the solution and by the expressivity results of \cite{OSX}. The same enriched space can then be coupled with different training objectives, including residual minimization and, for self-adjoint problems, energy minimization.

In the sections that follow we describe the proposed BL-NN architecture and demonstrate, through a variety of numerical examples, that it provides robust and accurate emulations over a wide range of singular perturbation regimes in one and two spatial dimensions.
The philosophy advocated in this paper is that whenever reliable asymptotic information about the solution is available, it should be incorporated directly into the architecture of the neural network rather than left for the optimization process to discover.

The rest of the article is organized as follows. Section \ref{sec:model} presents the model
problem and the regularity of its solution. Section \ref{sec:nn} gives the necessary 
definitions and results from the literature regarding NNs, and Section \ref{sec:SBL2NN} describes the proposed BL-NN. The results of numerical computations appear in Section \ref{sec:nr}, and finally Section \ref{sec:Concl} includes our conclusions.

We will denote by $I$ an open, bounded interval in $\mathbb{R}$, and by $L^{\infty}(I)$ the space of
essentially bounded functions on $I$, with norm $\Vert \cdot \Vert_{L^{\infty}(I)}$. The space of square-integrable
functions on $I$ will be denoted by $L^2(I)$, and its norm by $\Vert \cdot \Vert_{0,I}$. Finally, the letter $C$, with or without decorations, will denote a generic positive constant independent of any discretization
or singular perturbation parameters.

\section{Second order SPPs and the regularity of their solution}
\label{sec:model}
We consider the boundary value problem (BVP): find $u \in C^2(I) \cap C(\bar{I})$ such that
\begin{eqnarray}
-\e_1 u''(x) + \e_2  b(x)u'(x) + c(x) u(x) &=& f(x) \; , \; x \in I = (0,1) \; , \label{eq:de} \\
u(0) =  u(1)&=& 0, \label{eq:bc}
\end{eqnarray}
where $\e_1, \e_2 \in (0,1]$ and $b(x), c(x), f(x)$ are given analytic functions on
$\overline{I}=[0,1]$, satisfying $\forall \;x \in \overline{I}$,
$$
b(x) \geq \underline{b}>0, \quad c(x) \geq \underline{c}>0, \quad c(x)-\frac{\varepsilon_2}{2} b^{\prime}(x) \geq \underline{\delta} >0,
$$
for some positive constants $\underline{b}, \underline{c}, \underline{\delta} \in \mathbb{R}$, as well as
\begin{equation}
\label{eq:analytic}
\Vert f^{(n)} \Vert_{L^{\infty}(I)} \leq C_f K^n_f n! \: , \: \Vert b^{(n)} \Vert_{L^{\infty}(I)} \leq C_b K^n_b n! 
 \: , \: \Vert c^{(n)} \Vert_{L^{\infty}(I)} \leq C_c K^n_c n!
\end{equation}
$\forall \; n \in \mathbb{N}_0$,  where $C_f, K_f, C_b, K_b, C_c, K_c$, are positive constants independent of $\e_1$ and  $\e_2$.
(The superscript in parentheses denotes differentiation of order $n$.)

The structure of the solution to (\ref{eq:de})--(\ref{eq:bc}) depends on the roots of the characteristic
equation associated with the differential operator. For this reason, we let $\lambda_0(x), \lambda_1(x)$ be the solutions of the characteristic equation and set
\begin{equation}\label{mu}
\mu_0=-\max _{x \in[0,1]} \lambda_0(x), \; \mu_1=\min _{x \in[0,1]} \lambda_1(x),
\end{equation}
or equivalently,
$$
\mu_{0,1}=\min _{x \in[0,1]} \frac{\mp \varepsilon_2 b(x)+\sqrt{\varepsilon_2^2 b^2(x)+4 \varepsilon_1 c(x)}}{2 \varepsilon_1}.
$$
The values of $\mu_0, \mu_1$ determine the strength of the boundary layers and since $\left|\lambda_0(x)\right|< \left|\lambda_1(x)\right|$ the layer at $x=1$ is stronger than the layer at $x=0$. Essentially, there are three regimes, as seen in Table 1 below \cite{L}. Figure \ref{F0} shows the exact solution for $b(x)=c(x)=f(x)=1$, and different choices of $\e_1$ and  $\e_2$ corresponding to all three regimes.

\begin{table}[h]
\begin{tabular}{||cccc||}
\hline & & $\mu_0$ & $\mu_1$ \\
\hline \hline convection-diffusion & $\varepsilon_1 \ll \varepsilon_2=1$ & 1 & $\varepsilon_1^{-1}$ \\
\hline convection-reaction-diffusion & $\varepsilon_1 \ll \varepsilon_2^2 \ll 1$ & $\varepsilon_2^{-1}$ & $\varepsilon_2 / \varepsilon_1$ \\
\hline reaction-diffusion & $1 \gg \varepsilon_1 \gg \varepsilon_2^2$ or $\varepsilon_1 \approx \varepsilon_2^2$ & $\varepsilon_1^{-1 / 2}$ & $\varepsilon_1^{-1 / 2}$ \\
\hline
\end{tabular}
\caption{Different regimes based on the relationship between $\varepsilon_1$ and $\varepsilon_2$.}
\end{table}

\begin{figure}[h]
\begin{center}
\includegraphics[width=0.6\textwidth]{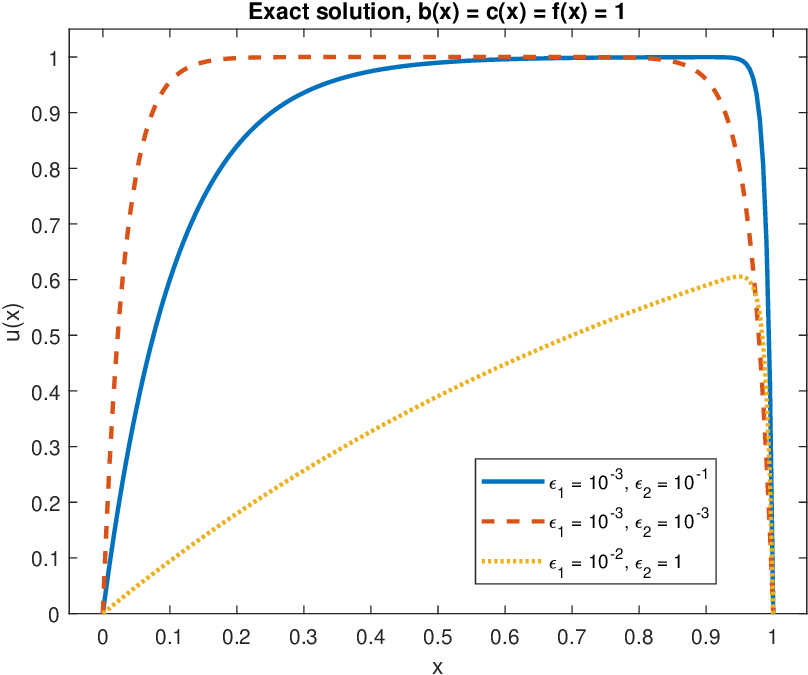} 
\end{center}
\caption{The exact solution of (\ref{eq:de})--(\ref{eq:bc}),  with $b(x)=c(x)=f(x)=1$.}
\label{F0}
\end{figure}

We will focus on the most interesting case, when  
$$
\e_1 \ll \e^2_2,
$$ 
thus we will have two different width boundary layers. The analytic regularity of the solution to (\ref{eq:de})--(\ref{eq:bc}) was studied in \cite{Irene, SX}, where it was shown that, in this case, the solution $u_{\e}$ may be decomposed as
\begin{equation}\label{decomp}
u(x) = u_S(x) + u_{BL}^{-}(x) + u_{BL}^{+}(x) + u_R(x),
\end{equation}
where $u_S$ is the smooth part, $u^{\pm}_{BL}$ are the boundary layers at the two endpoints,
and $u_R$ is the remainder. 
Moreover, there exist positive constants 
$K_1, \tilde{K}_S, \tilde{K}, \bar{K}, \tilde{\delta}, \beta$ independent of $\e_1$ and  $\e_2$, such that $\forall \; n \in \mathbb{N}_0$
and $x \in [0,1]$, there holds \cite{Irene, SX}
\begin{eqnarray*}
\Vert u^{(n)} \Vert_{L^{\infty}(I)} &\leq C & {K}_1^n \max\{ n, \mu_1 \}^{n}  \; ,  \label{Classical3} \\
\left\Vert u_{S}^{(n)}\right\Vert_{L^{\infty}(I)} &\leq C &  n! \tilde{K}_{S}^{n}\; ,   \label{uM1bound} \\
\left\vert \left(u_{BL}^{-}\right) ^{(n)}(x)\right\vert & \leq C  & 
\tilde{K}^{n}\mu_{0}^{n}e^{-\beta \mu_0 x}\; ,  \label{uBL1abound} \\
\left\vert \left( u_{BL}^{+}\right) ^{(n)}(x)\right\vert &\leq C & \bar{K}^{n}  \mu_1^n e^{-\beta \mu_1 (1-x)} \; ,   \label{uBL1bbound} \\
\left\Vert u_{R} \right\Vert _{{L^{\infty}( I)}} &\leq C & e^{-\tilde{\delta} /\varepsilon _{2}}.  \label{rM1bound}
\end{eqnarray*}

We point out that the above result says that the solution consists of a smooth part (analytic if the data are analytic), exponential layers of different widths which are strongly localized near the endpoints, and a negligible remainder.

\section{Neural Networks}
\label{sec:nn}
Following  \cite{PV,OPS}, we define a NN in terms of its \emph{weight matrices} and \emph{bias vectors}.
We distinguish between a neural network and the function it realizes,
called \emph{realization} of the NN, which is the composition of parameter-dependent affine transformations
and a nonlinear activation function.
We briefly recall the NN formalism from, e.g., \cite{OPS}.

\begin{definition}
\label{def:NeuralNetworks}
For $d,L\in\N$, a \emph{neural network $\Phi$} 
with input dimension $d \geq 1$ and number of layers $L\geq 1$, 
comprises a finite sequence of matrix-vector tuples, i.e.
\begin{align*}
\Phi = ((A_1,b_1),(A_2,b_2),\ldots,(A_L,b_L)),
\end{align*}
along with a nonlinear activation function $\varrho : \mathbb{R} \rightarrow \mathbb{R}$.
For $N_0 := d$ and \emph{numbers of neurons $N_1,\ldots,N_L\in\N$ per layer}, 
for all $\ell=1,\ldots, L$ it holds that
$A_\ell\in\R^{N_\ell \times N_{\ell-1} }$ and
$b_\ell\in\R^{N_\ell}$.

The \emph{realization} of $\Phi:\R^{N_0} \to \R^{N_L}$ as a map, is the function
\begin{align*}
\realiz(\Phi): \R^d\to\R^{N_L} : x \to x_L,
\end{align*}
where
\begin{align*}
x_0 & := x,
\\
x_\ell & := \varrho( A_\ell x_{\ell-1} + b_\ell ),
\qquad\text{ for }\ell=1,\ldots,L-1,
\\
x_L & := A_L x_{L-1} + b_L.
\end{align*}
Here, for $\ell=1,\ldots,L-1$, the activation function
$\varrho$ is applied componentwise: for
$y = (y_1,\ldots,y_{N_\ell})\in\R^{N_\ell}$ we denote
$\varrho(y) = ( \varrho (y_1), \ldots,
\varrho (y_{N_\ell}) )$.

We call the layers indexed by $\ell=1,\ldots,L-1$ \emph{hidden layers}, and
in those layers the activation function is applied.  No
activation is applied in the last layer of the NN. 

We refer to $\depth(\Phi) := L$ as the \emph{depth} of $\Phi$
and call $\size(\Phi) $ the \emph{size} of $\Phi$,
which is the number of nonzero components 
in the weight matrices $A_\ell$ and the bias vectors $b_\ell$.
Furthermore, we call $d$ and $N_L$ the
\emph{input dimension} and the \emph{output dimension}.
\end{definition}

If the desired accuracy is achieved by a NN with a fixed number of layers $L$  (usually at a low value), we refer to it as a \emph{shallow} NN. If, on the other hand, $L$ is increasing in order to achieve the desired accuracy, we refer to the network as a \emph{deep} NN.  The ``calculus'' of NNs, i.e. how they can be combined, is described in, e.g., \cite[Sec. 2]{OSX}. 

In this article, we will utilize the $\tanh$ activation function
$$
	\rho: \R \to \R: x \mapsto \tanh(x) = \frac{e^x - e^{-x}}{e^x + e^{-x}}.
$$
A nice feature of this activation function is that it may be used for the emulation of $e^{-x}$ within a
prescribed tolerance $TOL$. In particular, the 2 layer NN
$$
\Phi_{\exp}^{\tanh} := \left( (A_1, b_1) , (A_2, b_2)  \right),
$$
with $A_1 = \frac{1}{2}, b_1 = \frac{x_*}{2}, A_2 = -\frac{e^{x_*}}{2}, b_2 =\frac{e^{x_*}}{2}, x_* = \log(2/TOL)$, has realization
$$
R[\Phi_{\exp}^{\tanh}] (x) = \frac{2}{2e^x + TOL}.
$$
Moreover, $\tanh$ NNs can emulate smooth functions using a fixed number of layers \cite{Mishra}.
Hence, \(\tanh\) NNs are well suited for approximating the analytic component of solutions to SPPs, while the preceding approximation result also shows that the exponential layer profiles can, in principle, be represented efficiently by neural networks. In the computational architecture developed in the next section, however, we exploit the available asymptotic information more directly by augmenting a shallow \(\tanh\) network with explicit exponential layer features.

\section{Emulation of the solution to SPPs}
\label{sec:SBL2NN}
In terms of expressivity there exists a NN for emulating the solution to reaction-diffusion SPPs with constant coefficients, within a prescribed tolerance \cite[Theorem 14]{OSX}. For the problem (\ref{eq:de})--(\ref{eq:bc}), with constant coefficients, we have the following.

\begin{proposition}\label{prop:NNbl}
Let $p \in \mathbb{N}$ be given and assume $u$ is the solution to (\ref{eq:de})--(\ref{eq:bc}), with $b(x)=b>0, c(x) = c >0$. Assuming $f$ is analytic, there exists a $\tanh$ NN $\Phi^{p}_{{SPP}}$ with size $M(\Phi^{p}_{{SPP}})$ and length $L(\Phi^{p}_{{SPP}})$, such that for all $\e_1 \ll \e_2^2 \ll 1$, the following estimate holds
$$
\Vert u - R[\Phi^{p}_{SPP}] \Vert_{L^{\infty}(I)} \leq C \left( e^{-\beta p} + e^{-\tilde{\delta} / \varepsilon_2} \right),
$$
for some positive constants $C, \beta, \tilde{\delta}$. Moreover, $M(\Phi^p_{SPP}) = O(p)$ and $L(\Phi^{p}_{{SPP}})= \ceil{3 \beta p/2}$.
\end{proposition}
\begin{proof}
We utilize the decomposition (\ref{decomp}) and construct separate NNs for each component, except for the remainder which is already exponentially small.

For the smooth (analytic) part $u_S$, we use \cite[Cor. 5.8]{Mishra} to deduce the existence of a $\tanh$ NN, say $\Phi^{\tanh}_S$, such that 
$$
\left\Vert u_S - R\left[\Phi^{\tanh}_{S}\right] \right\Vert_{L^{\infty}(I)} \leq C e^{-\beta p},
$$
with $M(\Phi^{\tanh}_S) = O(p)$, and $L(\Phi^{\tanh}_S) = \ceil{3 \beta p/2}$.

Since the coefficients are constant, the leading layer profiles are explicitly known, hence the shallow NN $\Phi^{\tanh}_{\text{exp}}$ emulates them within a given tolerance. In fact,
$$
\left\Vert u^{\pm}_{BL} - R^{\pm}\left[\Phi^{\tanh}_{\text{exp}}\right] \right\Vert_{L^{\infty}(I)} \leq C e^{-\beta p},
$$
where $R^{\pm}$ denotes the realization of the network at each endpoint.

We combine the three subnetworks in parallel (see, e.g., Proposition 2.3 in \cite{OPS}), and we obtain $\Phi^p_{SPP}$. The shallower exponential subnetworks are extended by identity layers so that the parallelization has the depth of the deepest component. Then,
$$
\begin{aligned}
\left\|u-R\left[\Phi_{\mathrm{SPP}}^p\right]\right\|_{L^{\infty}(I)} \leq & \left\|u_S-R\left[\Phi^{\tanh}_S\right]\right\|_{L^{\infty}(I)}+\left\|u_{BL}^{-}-R^{-}\left[\Phi_{\text{exp}}^{\tanh}\right]\right\|_{L^{\infty}(I)} \\
& +\left\|u_{BL}^{+}-R^{+}\left[\Phi_{\text{exp}}^{\tanh}\right]\right\|_{L^{\infty}(I)}+\left\|u_R\right\|_{L^{\infty}(I)} \\
\leq & C\left(e^{-\beta p}+e^{-\tilde{\delta} / \varepsilon_2}\right) .
\end{aligned}
$$
The resulting NN has size $M(\Phi_{SPP}^p)$ equal to the size of the three subnetworks, i.e.~$M(\Phi_{SPP}^p)=O(p)$,
and depth $L(\Phi_{SPP}^p) = \ceil{3 \beta p/2}=O(p)$.
\end{proof}

\begin{remark}
Proposition~\ref{prop:NNbl} is an expressivity statement for a family whose size and depth may grow with $p$; it is more general than the one-hidden-layer enriched architecture used in the computations below.
\end{remark}

\begin{remark}
The assumption of constant coefficients in the above proposition is made so that the layers will be explicitly known,  hence the $\tanh$ NN $\Phi^{\tanh}_{\text{exp}}$ can emulate them within the desired tolerance. In our implementation, however, we will choose the constants in the exponential layers to be trainable. Treating the exponential decay rates as trainable provides a practical mechanism for accommodating layer scales that are not known exactly, as occurs for variable coefficients.
\end{remark}

\subsection{A boundary layer resolving NN: BL-NN}
\label{sec:BLNN}
Guided by the asymptotic expansions, we simply augment/enhance the trial solution with exponential functions. For the BVP (\ref{eq:de})--(\ref{eq:bc}), the boundary layers behave like
$$
u_{BL}^{-} \sim e^{-\beta \mu_0 x} \; , \; u_{BL}^{+} \sim e^{-\gamma \mu_1 (1-x)},
$$
where $\beta, \gamma$ are positive $O(1)$ constants, which will be treated as trainable. Accordingly, in the computations we use a shallow $\tanh$ NN (one hidden layer with $n$ neurons) and augment it directly with the two exponential features. We use as trial function
\begin{equation}\label{uNN}
u_{NN}(x)= \sum_{j=1}^n a_j \tanh \left(w_j x+b_j\right)+d_L e^{-\beta \mu_0 x} +d_R e^{-\gamma \mu_1 (1-x)},
\end{equation}
in which $a_j, w_j , b_j, \beta, \gamma$ are learnable quantities, $\mu_0, \mu_1$ are given by (\ref{mu}), and $d_L, d_R$ are chosen so that the boundary conditions are enforced strongly -- actually, the parameters $d_L, d_R$ are updated at each step so that the boundary conditions are satisfied. In a sense, this resembles the \emph{enrichment} of, e.g., finite element spaces (see, e.g., \cite{Kellogg} and the references therein). In fact, the architecture is
$$
u_{NN} = u^N_S + \sum_i \zeta_i E_i,
$$
where $u^N_S \in V_{NN} = \text{span} \{ \tanh(w_j x + b_j)\}$, $E_i$ are layer functions and $\zeta_i \in \mathbb{R}$. Thus, from an approximation-theoretic point of view, the essential construction is the enriched space
$$
u_{NN} \in V_{NN} \oplus \text{span}\{ E_i \}.
$$

\begin{remark}
The BL-NN can be summarized independently of the particular training objective as follows. Starting from an available asymptotic decomposition, we (i) identify the location and width of each layer, (ii) introduce corresponding layer features $E_i$, (iii) enrich a shallow $\tanh$ approximation space by $\operatorname{span}\{E_i\}$, (iv) enforce the boundary conditions strongly whenever the chosen parametrization permits this, and (v) determine the learnable parameters by minimizing a suitable residual or energy functional on layer-adapted training points. This separates the roles of the approximation space, sampling strategy, and optimization algorithm.
\end{remark}

\begin{remark}
The strong enforcement of the boundary conditions is not essential. We have performed numerical experiments in which we used \emph{Nitsche's} method \cite{Nitsche} for weakly imposing the boundary conditions, with almost identical results.
\end{remark}

\subsubsection{Shallow Ritz}
\label{shallowritz}
In the Ritz method, we minimize an energy functional with respect to the weights and biases, and it is most natural for self-adjoint problems. 
For the problem (\ref{eq:de})--(\ref{eq:bc}), we first define the integrating factor
$$
\mu(x)=\exp \left(-\int_0^x \frac{\varepsilon_2 b(s)}{\varepsilon_1} d s\right),
$$
so that 
$$
-\varepsilon_1\left(\mu(x) u^{\prime}(x)\right)^{\prime}+\mu(x) c(x) u(x)=\mu(x) f(x) .
$$
We then minimize the functional
$$
\mathcal{J}(u_{NN}):= \frac{1}{2} \int_0^1 \mu(x) \left(\e_1 (u_{NN}^{\prime}(x))^2+c(x) u_{NN}^2(x)\right)dx -\int_0^1 \mu(x) f(x) u_{NN}(x) dx, 
$$
with respect to the learnable parameters, using a \emph{Shishkin} mesh \cite{Shishkin2} for generating the training points. We approximate all integrals using (composite) Gaussian quadrature. We note that, in a strongly convection-dominated regime, the integrating factor may itself vary on an extreme scale and can lead to poor numerical conditioning. For this reason the Ritz formulation is most attractive in the self-adjoint reaction-diffusion setting considered below; the nonsymmetric convection-reaction-diffusion examples are treated by residual minimization.

\subsubsection{Shallow PINN}
\label{shallowpinn}
In Physics Informed Neural Networks (PINNs) \cite{LLF, RPK}, a collocation residual is minimized with respect to the weights and biases. For the problem (\ref{eq:de})--(\ref{eq:bc}), let
\[
r(x;u_{NN})=- \e_1 u''_{NN}(x) + \e_2 b(x) u'_{NN}(x) + c(x)u_{NN}(x)-f(x),
\]
denote the differential residual. At collocation points $\{x_i\}_{i=1}^N$ we minimize the mean-squared residual
\begin{equation}\label{PINNloss}
\mathcal{L}(u_{NN}) = \frac{1}{N}\sum_{i=1}^N \left|r(x_i;u_{NN})\right|^2.
\end{equation}
The boundary conditions are not included as penalty terms in the one-dimensional experiments because they are imposed strongly through the parametrization in (\ref{uNN}). We again use a \emph{Shishkin} mesh \cite{Shishkin2} to generate the collocation points. Thus the layer-adapted sampling and the exponential enrichment are separate ingredients of the method.

\section{Numerical results}
\label{sec:nr}
In this section we present the results of numerical computations for second order SPPs, in one and two dimensions. 
We will provide solution plots for relatively large values of the singular perturbation parameters, in order for the phenomena to be visible.  If an exact solution is available, we will use it for comparison, and if not we will compare the emulation with a numerical solution obtained using the FEM.

The numerical experiments are designed to assess whether the enriched architecture remains accurate as the perturbation parameters decrease. Over the tested parameter ranges, the $L^\infty$, $L^2$, and energy errors remain nearly parameter-independent, while the unweighted $H^1$ error may grow as the layers become thinner. We use both residual- and energy-based training objectives, depending on the differential operator, but employ the stochastic Adam optimization algorithm \cite{adam} in all examples. Accordingly, the computations support robustness with respect to the tested perturbation parameters and to these two choices of training objective. The architecture also extends naturally to two dimensions (cf. Examples 3 and 4).

\vspace{0.2cm}

\noindent
\textbf{Example 1:} 
\vspace{0.2cm}

\noindent
We consider the BVP (\ref{eq:de})--(\ref{eq:bc}), with $\e_1\ll \e^2_2$, and $ b(x)=c(x)=f(x)=1$, i.e.~a reaction-convection-diffusion SPP, with an exact solution available. We will be using the shallow PINN of Section \ref{shallowpinn}, with $u_{NN}$ given by (\ref{uNN}).

First we perform a systematic study to choose the hyperparameters. In particular, we vary one parameter while keeping the others fixed, and the results are summarized in Table \ref{table2}, in which the maximum absolute pointwise error between the exact solution and the emulation is reported each time.

\begin{table}[htbp]
\centering
\caption{Hyperparameter study for Example 1 with
$\varepsilon_1=10^{-5}$ and $\varepsilon_2=10^{-2}$.  (While varying the one parameter, the rest were fixed at the values shown in bold.)}
\label{table2}

\renewcommand{\arraystretch}{1.15}
\setlength{\tabcolsep}{10pt}

\begin{tabular}{cc}

\begin{tabular}{cc}
\multicolumn{2}{c}{\textbf{Neurons}}\\
\toprule
Number & $L^\infty$ error\\
\midrule
10 & $2.2574\times10^{-3}$\\
\textbf{20} & $\mathbf{2.1719\times10^{-4}}$\\
30 & $1.2752\times10^{-3}$\\
\bottomrule
\end{tabular}

&

\begin{tabular}{cc}
\multicolumn{2}{c}{\textbf{Learning rate}}\\
\toprule
Rate & $L^\infty$ error\\
\midrule
$10^{-1}$ & $4.7404\times10^{-3}$\\
$\mathbf{10^{-2}}$ & $\mathbf{2.3140\times10^{-4}}$\\
$10^{-3}$ & $1.8486\times10^{-3}$\\
\bottomrule
\end{tabular}

\\[8mm]

\begin{tabular}{cc}
\multicolumn{2}{c}{\textbf{Training epochs}}\\
\toprule
Epochs & $L^\infty$ error\\
\midrule
1000 & $6.2350\times10^{-3}$\\
2000 & $7.7126\times10^{-4}$\\
3000 & $2.9866\times10^{-4}$\\
4000 & $2.6370\times10^{-4}$\\
5000 & $2.3716\times10^{-4}$\\
\textbf{6000} & $\mathbf{2.1719\times10^{-4}}$\\
7000 & $2.3142\times10^{-4}$\\
\bottomrule
\end{tabular}

&

\begin{tabular}{cc}
\multicolumn{2}{c}{\textbf{Shishkin points}}\\
\toprule
Points & $L^\infty$ error\\
\midrule
20  & $2.1720\times10^{-4}$\\
\textbf{40}  & $\mathbf{2.1719}\times\mathbf{10^{-4}}$\\
80  & $2.1719\times10^{-4}$\\
\bottomrule
\end{tabular}

\end{tabular}

\end{table}

Guided by the results in Table \ref{table2}, we choose for the remaining experiments in this example:
$$\text{Epochs} = 6000, \text{Neurons} = 20, \text{Shishkin points} = 40, \text{Learning rate} = 0.01.$$

\noindent
Now, for the proposed BL-NN, we show in Table \ref{table_BLNNex3} different error measures for various combinations of $\e_1 \ll \e_2^2$, and in Figure \ref{fig_BLNN4} we show the emulation and exact solution, as well as the absolute pointwise maximum error between them, for $\e_1 = 10^{-3}, \e_2 = 10^{-1}$.  The lengthiest run took about 1 minute.
\begin{table}[ht]
	\centering
	\caption{Error measures for the BL-NN emulation, for Example 1.}
	\begin{tabular}{lccccc}
		\hline
		$\e_1$ & $\e_2$ & ${L^{\infty}}$ error & ${L^{2}}$ error & $ H^{1}$ error & Energy error \\
		\hline
		$10^{-3}$ & $10^{-1}$ & $1.2525\times10^{-3}$ & $ 8.0309\times10^{-4}$ & $1.1757\times10^{-2}$ & $8.8461\times10^{-4}$\\ 
		$10^{-5}$& $10^{-2}$ & $1.1438\times10^{-3}$ & $ 3.2610\times10^{-4}$ & $2.6165\times10^{-2}$ & $3.3643\times10^{-4}$\\
		$10^{-7}$ & $10^{-3}$ & $1.3679\times10^{-3}$ & $ 3.6994\times10^{-4}$ & $9.4642\times10^{-2}$ & $3.7114\times10^{-4}$\\ 
		$10^{-9}$ & $10^{-4}$ & $1.3974\times10^{-3}$ & $ 3.7452\times10^{-4}$ & $3.0363\times10^{-1}$ & $3.7465\times10^{-4}$\\ 
		$10^{-11}$ & $10^{-5}$ & $9.4782\times10^{-4}$ & $ 2.3284\times10^{-4}$ & $3.1633\times10^{-1}$ & $3.7834\times10^{-4}$\\ 
 \hline
	\end{tabular}
	\label{table_BLNNex3}
\end{table}

\begin{figure}[h!]
\begin{center}
\includegraphics[width=0.475\textwidth]{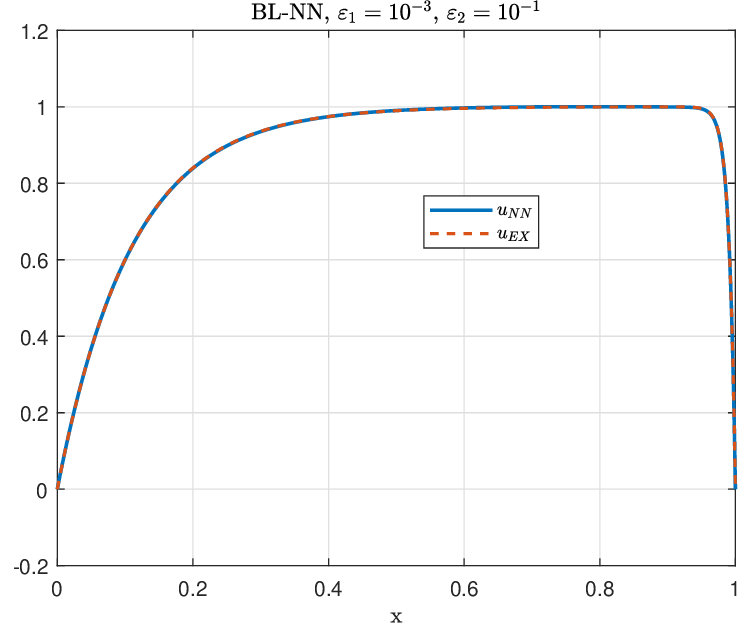}
\mbox{ }
\includegraphics[width=0.475\textwidth]{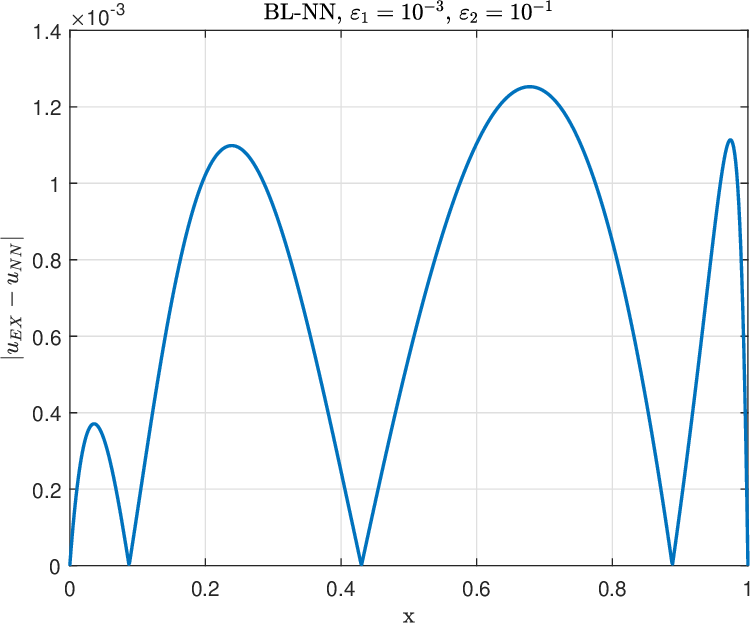}
\end{center}
\caption{Comparison of the exact solution and BL-NN approximation with $\e_1=10^{-3}, \e_2 = 10^{-1}$, for Example 1.}
\label{fig_BLNN4}
\end{figure}

From Table \ref{table_BLNNex3}, we see that the method is robust and emulates the solution with $O(10^{-3})$ accuracy, over all tested parameter combinations. In Figure \ref{tr1} we show the training history for $\varepsilon_1 = 10^{-7}, \varepsilon_2 = 10^{-3}$ as well as $\varepsilon_1 = 10^{-11}, \varepsilon_2 = 10^{-5}$. (Other values of $\varepsilon_1, \varepsilon_2$ gave similar results.)

\begin{figure}[h!]
\begin{center}
\includegraphics[width=0.475\textwidth]{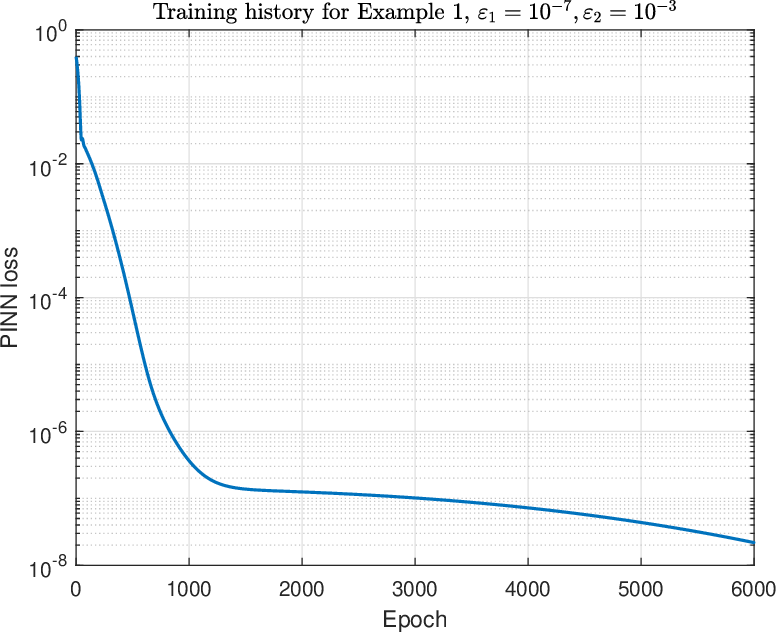}
\mbox{ }
\includegraphics[width=0.475\textwidth]{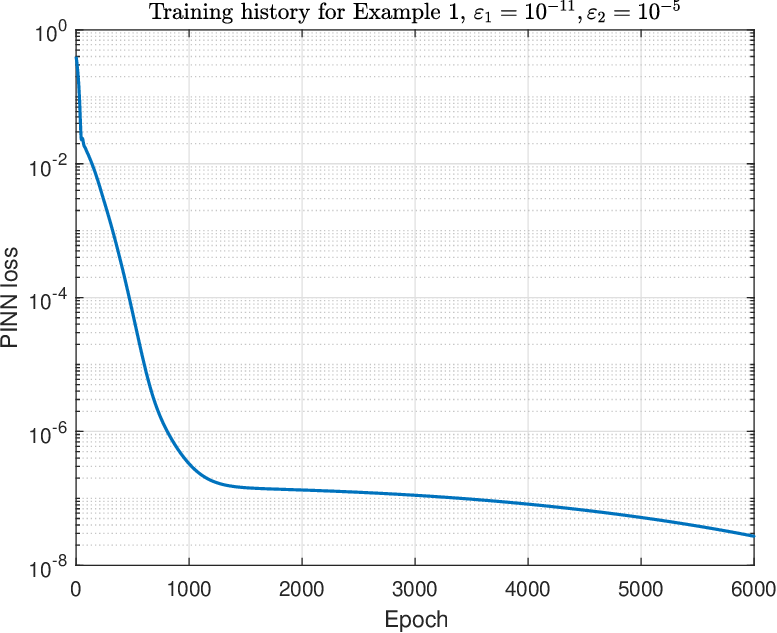}
\end{center}
\caption{Training history for Example 1.}
\label{tr1}
\end{figure}

We next want to see the effect of \emph{not including} the enrichments in (\ref{uNN}). First we use no enrichment, then we use just one on the left, and then just one on the right. Note that the case of no enrichments corresponds to using a standard PINN one layer network. For the representative values $\e_1=10^{-5}, \e_2 = 10^{-2}$, we compute the emulation, and in Figure \ref{enr} we show the exact solution along with $u_{NN}$. It is evident from the figure that using both enrichments yields the best results. Other values of the singular perturbation parameters gave similar results. So for the remainder of this section we will include all appropriate enrichments.

\begin{figure}[h!]
\begin{center}
\includegraphics[width=0.475\textwidth]{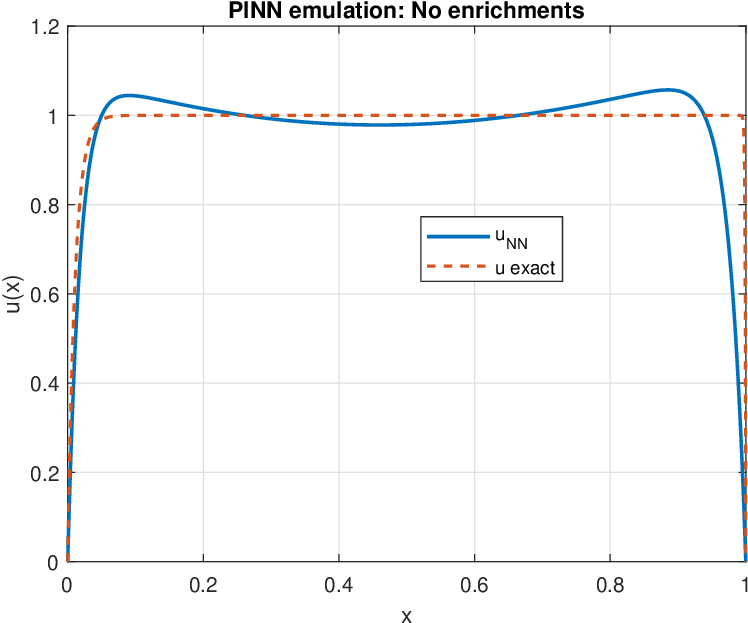}
\mbox{ }
\includegraphics[width=0.475\textwidth]{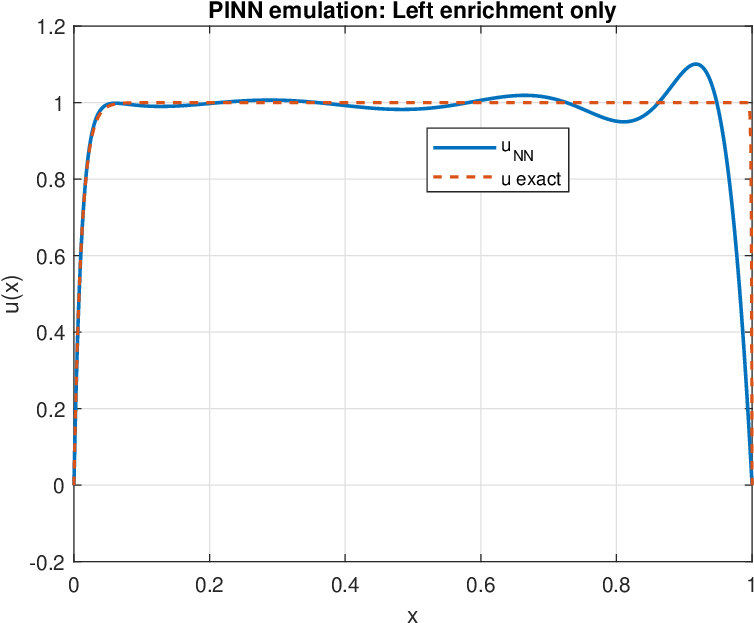}
\mbox{ }
\includegraphics[width=0.475\textwidth]{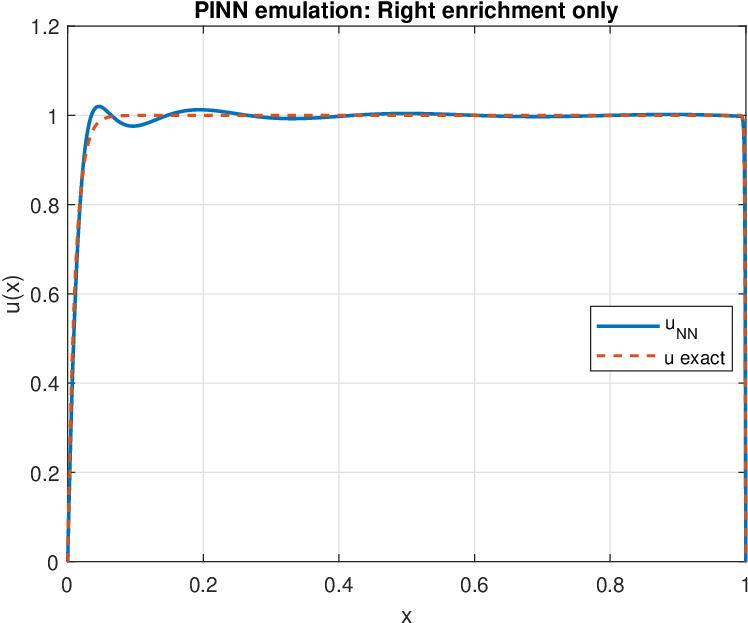}
\mbox{ }
\includegraphics[width=0.475\textwidth]{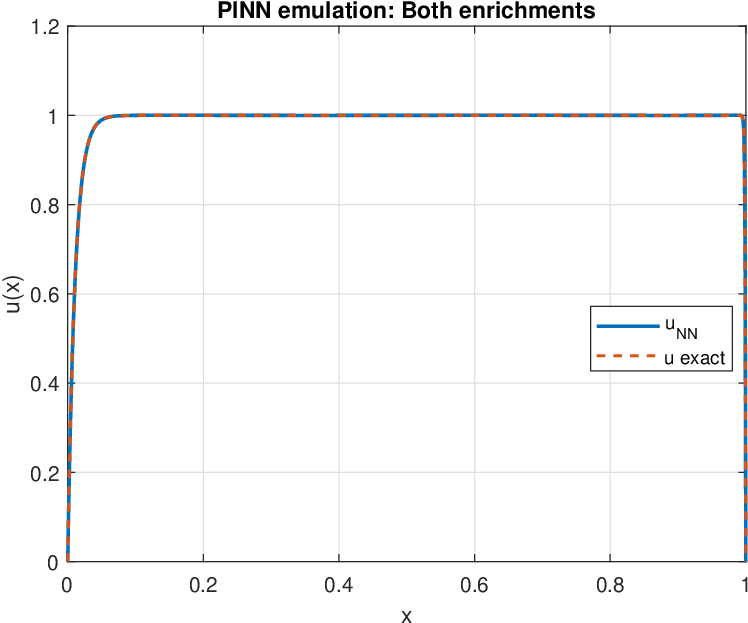}
\mbox{ }
\end{center}
\caption{Effect of using enrichments, $\e_1=10^{-5}, \e_2 = 10^{-2}$, for Example 1.}
\label{enr}
\end{figure}

\vspace{1cm}

\noindent
\textbf{Example 2:} 
\vspace{0.2cm}

\noindent
We consider the BVP (\ref{eq:de})--(\ref{eq:bc}), with $\e_1= \e^2, b(x)=0, c(x)=1+x^2, f(x) = e^{-x^2}$, i.e.~a reaction-diffusion SPP, \emph{without} an available exact solution. We will compare $u_{NN}$ with an approximation $u_{FEM}$ obtained using the $hp$ version of the FEM. Since we have a self-adjoint problem, we will use the shallow Ritz method of Section \ref{shallowritz}. With this example we illustrate that even when an exact solution is not available (e.g., to guide the choices of the hyperparameters), the proposed BL-NN gives very satisfactory results. Hence we will not perform an ablation study, nor will we consider the effect of not using all the enrichments. With (\ref{uNN}) as a trial solution, and with the same hyperparameters as in Example 1, we show in
Figure \ref{fig_BLNN1} (left) the emulation $u_{NN}$ (for relatively large values of $\e$). It appears qualitatively correct, given that we expect (same width) layers at the two endpoints. 
\begin{figure}[h!]
\begin{center}
\includegraphics[width=0.5\textwidth]{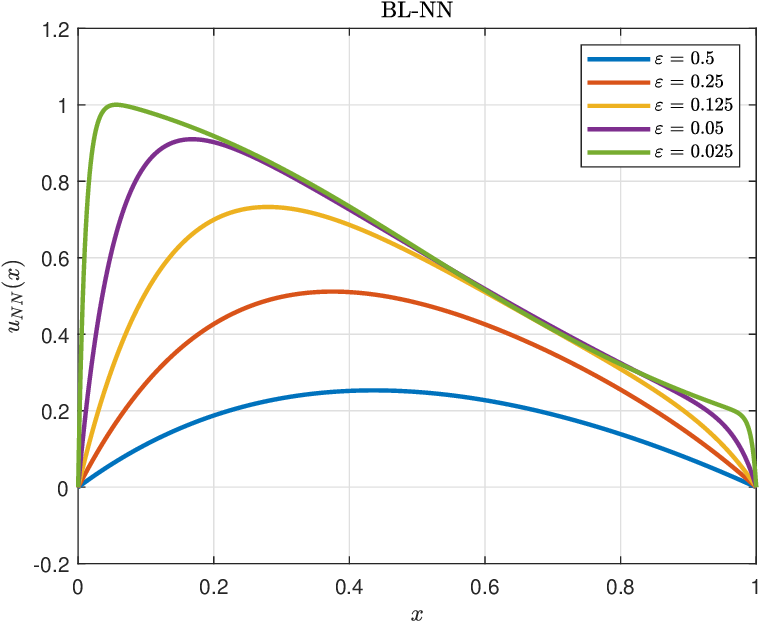}
\mbox{ }
\includegraphics[width=0.475\textwidth]{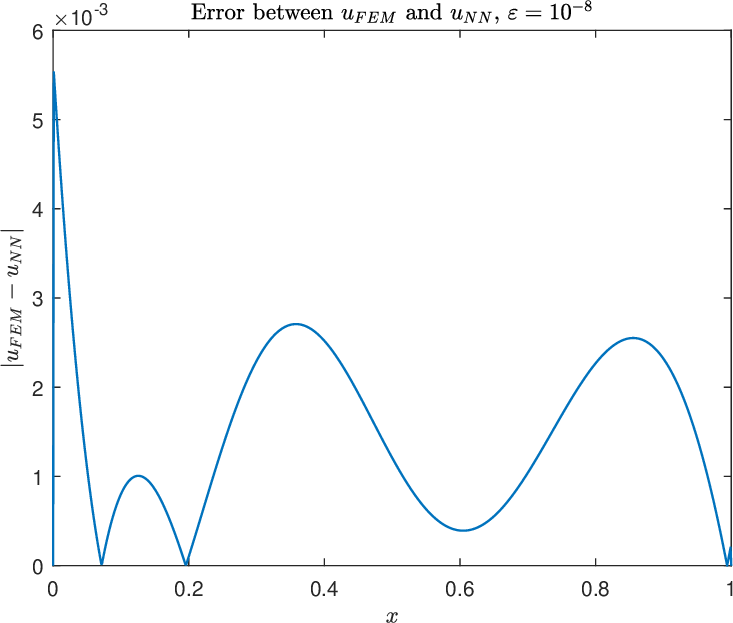}
\end{center}
\caption{Left: emulated solution using the BL-NN, for Example 2, with unknown exact solution. Right: error between the emulation and the $hp$-FEM approximation with polynomials of degree $p=15$.}
\label{fig_BLNN1}
\end{figure}

In order to assess the accuracy of the method, we compare $u_{NN}$ with an approximation obtained with the $hp$ FEM on the \emph{Spectral Boundary Layer} mesh \cite{MXO}. In Figure \ref{fig_BLNN1} (right) we show the pointwise absolute error between the two, for $\e=10^{-8}$, which is of $O(10^{-3})$. This provides an independent numerical reference suggesting that the BL-NN solution is close to a robust $hp$-FEM approximation. A direct comparison of computational efficiency would require matched implementations and timings at comparable accuracy and is therefore not claimed here.

\newpage

\noindent
\textbf{Example 3:} 
\vspace{0.2cm}

\noindent
As a two dimensional extension of the proposed method, we consider the convection-reaction-diffusion problem: find $u(x,y) \in C^2(\Omega) \cap C(\overline{\Omega})$ such that
\begin{eqnarray}
-\e_1 \nabla^2 u + \e_2 b \frac{\partial u}{\partial x} + c(x,y) u(x,y) &=& f(x,y) \; \text{ in } \; \Omega = (0,1)^2, \label{pde_2D}\\
u|_{\partial \Omega} &=& 0, \label{bc_2D}
\end{eqnarray}
where $\e_1 < \e_2 \in (0,1]$ are parameters that may approach zero, $b$ is a positive constant, and $c(x,y), f(x,y)$ are given sufficiently smooth functions, with $c(x,y) \ge \underline c>0 \; \forall \; (x,y) \in \overline{\Omega}$. Assuming $\e_1 \ll \e_2^2$, the solution $u$ will feature different width exponential layers at $x=0$ and $x=1$, as well as characteristic/parabolic layers at $y=0$ and $y=1$. In particular, we expect the solution to have an $O(\mu^{-1}_0)$ layer at $x=0$ and an $O(\mu^{-1}_1)$ layer at $x=1$, where $\mu_0, \mu_1$ are given by (\ref{mu}). A characteristic layer of width $O(\sqrt{\e_1})$ will also appear at $y=0$ and $y=1$. See, e.g., \cite{Irene} for a decomposition and regularity results. For the example we choose $b =1, c(x,y)=2$,
and we calculate $f(x,y)$, so that the exact solution is
$$
u(x, y)=\left(1-e^{-x / \varepsilon_2}\right)\left(1-e^{-\frac{(1-x)}{\varepsilon_1}} \right)\left(1-e^{-y / \sqrt{\varepsilon_1}}\right)\left(1-e^{-\frac{1-y}{\sqrt{\varepsilon_1}}}\right) .
$$

The expected asymptotic structure may be viewed schematically as a smooth component, two exponential $x$-layers, two characteristic $y$-layers, and four corner-layer contributions. Motivated by these components, we extend the one-dimensional architecture by combining the corresponding edge-layer and corner-layer features, and choose as trial solution
\begin{eqnarray*}
u_{NN}(x,y) &=& \sum_{j=1}^n a_j \tanh(w^x_j x + w^y_j y + b_j) + d_1 e^{-\gamma_1 x \mu_0} + d_2 e^{-\gamma_2 (1-x) \mu_1} + d_3 e^{-\gamma_3 y/{\sqrt{\e_1}} }\\
&+&  d_4 e^{-\gamma_4 (1-y)/\sqrt{\e_1}}+ d_5 e^{-\gamma_5 (x \mu_0 + y /\sqrt{\e_1})} + d_6 e^{-\gamma_6 ((1-x) \mu_1 + y /\sqrt{\e_1})}\\ 
&+& d_7 e^{-\gamma_7( x \mu_0 + (1-y) /\sqrt{\e_1})} + d_8 e^{-\gamma_8 ( (1-x)\mu_1 + (1-y) /\sqrt{\e_1})},
\end{eqnarray*}
where $a_j, w_j^x, w_j^y, b_j$, $j=1,\ldots,n$, and $\gamma_{\ell}$, $\ell=1,\ldots,8$, are trainable parameters, while $d_i$, $i=1,\ldots,8$, are the coefficients of the layer enrichments. In more than one dimension, the exponential ansatz alone does not automatically enforce $u=0$ along every boundary edge; exact strong enforcement therefore requires a separate boundary-correction parametrization. We utilize a transfinite (Coons) boundary projector \cite{Coons}:
\begin{eqnarray*}
   u &=& v - \left[(1-x)v(0,y)+xv(1,y)+(1-y)v(x,0)+yv(x,1)\right] \\
         &+& [(1-x)(1-y)v(0,0)+x(1-y)v(1,0) + (1-x)yv(0,1)+xyv(1,1)].
\end{eqnarray*}
The derivatives of the boundary projector are included analytically in the residual evaluation. 
We use a PINN loss, since the problem is not self-adjoint, and we first conduct a systematic study to choose the hyperparameters, for $\e_1 = 10^{-5}, \e_2 = 10^{-2}$, with the results shown in Table \ref{table4}. For the remaining tests we use the following common baseline configuration (chosen as a compromise between accuracy and cost rather than as the minimizer of every one-at-a-time study):
$$
\text{Epochs} = 6000, \text{Neurons} = 30, \; \text{Shishkin points} = 16, \; \text{Learning rate} = 0.01.
$$

\begin{table}[htbp]
\centering
\caption{Hyperparameter study for Example 3 with
$\varepsilon_1=10^{-5}$ and $\varepsilon_2=10^{-2}$.  (While varying the one parameter, the rest were fixed at the values shown in bold.)}
\label{table4}

\renewcommand{\arraystretch}{1.15}
\setlength{\tabcolsep}{10pt}

\begin{tabular}{cc}

\begin{tabular}{cc}
\multicolumn{2}{c}{\textbf{Neurons}}\\
\toprule
Number & $L^\infty$ error\\
\midrule
10 & $9.9334\times10^{-3}$\\
20 & ${1.1917\times10^{-3}}$\\
\textbf{30} & $\mathbf{6.6098\times10^{-4}}$\\
40 & $9.3678\times10^{-4}$\\
\bottomrule
\end{tabular}

&

\begin{tabular}{cc}
\multicolumn{2}{c}{\textbf{Learning rate}}\\
\toprule
Rate & $L^\infty$ error\\
\midrule
$10^{-1}$ & $1.7726\times10^{-3}$\\
$\mathbf{10^{-2}}$ & $\mathbf{6.6098\times10^{-4}}$\\
$10^{-3}$ & $1.4583\times10^{-2}$\\
\bottomrule
\end{tabular}

\\[8mm]

\begin{tabular}{cc}
\multicolumn{2}{c}{\textbf{Training epochs}}\\
\toprule
Epochs & $L^\infty$ error\\
\midrule
1000 & $1.5792\times10^{-2}$\\
2000 & $7.1778\times10^{-3}$\\
3000 & $3.9920\times10^{-3}$\\
4000 & $2.5491\times10^{-3}$\\
5000 & $1.6188\times10^{-3}$\\
\textbf{6000} & $\mathbf{6.6097\times10^{-4}}$\\
7000 & $4.9162\times10^{-4}$\\
\bottomrule
\end{tabular}

&

\begin{tabular}{cc}
\multicolumn{2}{c}{\textbf{Shishkin points}}\\
\toprule
Points & $L^\infty$ error\\
\midrule
4  & $1.8931\times10^{-3}$\\
8  & $9.5260\times10^{-4}$\\
\textbf{16} & $\mathbf{9.1930\times10^{-4}}$\\
32 & $9.3679\times10^{-4}$\\
\bottomrule
\end{tabular}

\end{tabular}

\end{table}

\begin{table}[ht]
	\centering
	\caption{Error measures for the BL-NN emulation, for Example 3.}
	\begin{tabular}{lccccc}
		\hline
		$\e_1$ & $\e_2$ & ${L^{\infty}}$ error & ${L^{2}}$ error & $ H^{1}$ error & Energy error \\
		\hline
		$10^{-3}$ & $10^{-1}$ & $4.4577 \times 10^{-4}$ & $ 1.6112 \times 10^{-4}$ & $2.8410 \times 10^{-3}$ & $2.4481 \times 10^{-4}$\\ 
		$10^{-5}$& $10^{-2}$ & $4.8112 \times 10^{-4}$ & $ 1.9654 \times 10^{-4}$ & $1.6510 \times 10^{-2}$ & $2.8270 \times 10^{-4}$\\
		$10^{-7}$ & $10^{-3}$ & $9.4217\times 10^{-4}$ & $ 3.4012\times 10^{-4}$ & $2.0946\times 10^{-1}$ & $4.8554\times 10^{-4}$\\ 
		$10^{-9}$ & $10^{-4}$ & $9.4957 \times 10^{-4}$ & $ 3.5581 \times 10^{-4}$ & $2.3121 \times 10^{-0}$ & $5.0847\times 10^{-4}$\\ 
		$10^{-11}$ & $10^{-5}$ & $9.3123 \times 10^{-4}$ & $ 3.5551\times 10^{-4}$ & $2.3140\times 10^{+1}$ & $5.0806\times 10^{-4}$\\ 
 \hline
	\end{tabular}
	\label{tableEx3_errors}
\end{table}

Table \ref{tableEx3_errors} shows various error measures for values of $\e_1 \ll \e_2^2$. Over the tested parameter range the maximum-norm error remains below $10^{-3}$, while the $L^2$ and energy errors are likewise nearly parameter-independent. In contrast, the unweighted $H^1$ error grows substantially as the perturbation parameters decrease. This is consistent with the layer structure: derivatives inside increasingly thin layers scale inversely with the layer width, whereas the energy norm weights the derivative contribution by the diffusion parameter. The lengthiest run took about 2 minutes. In Figure \ref{fig_ex3} we show the emulated solution, as well as the absolute error, for $\e_1 = 10^{-3}, \e_2 = 10^{-1}$, and in Figure \ref{fig_ex3_train} we show the training history. It is evident that the proposed method generalizes to two dimensions with satisfactory results.

 \begin{figure}[h!]
\begin{center}
\includegraphics[width=0.5\textwidth]{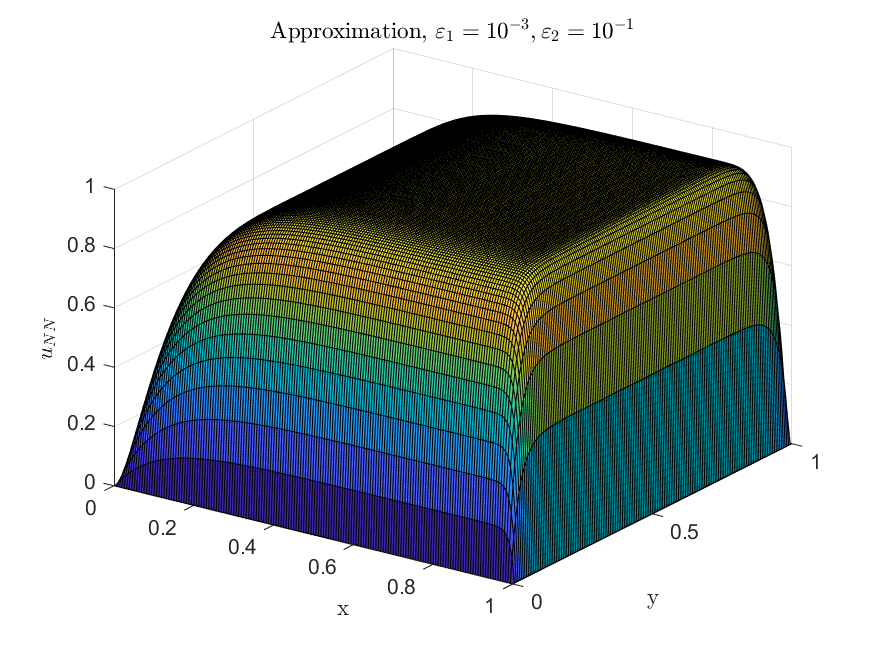}
\mbox{ }
\includegraphics[width=0.475\textwidth]{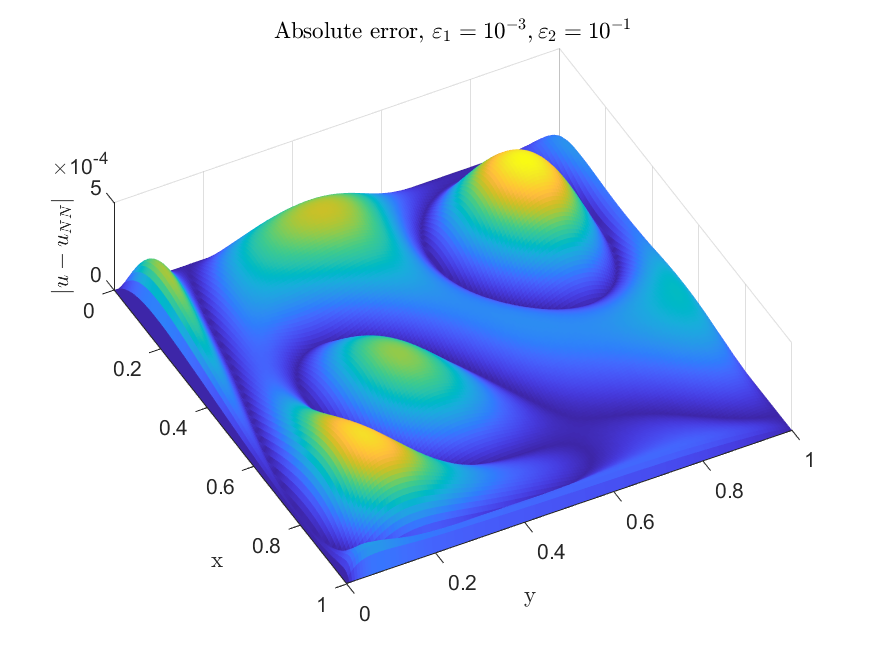}
\end{center}
\caption{Left: emulated solution using the BL-NN, for Example 3. Right: absolute error between $u_{NN}$ and the exact solution.}
\label{fig_ex3}
\end{figure}

 \begin{figure}[h!]
\begin{center}
\includegraphics[width=0.5\textwidth]{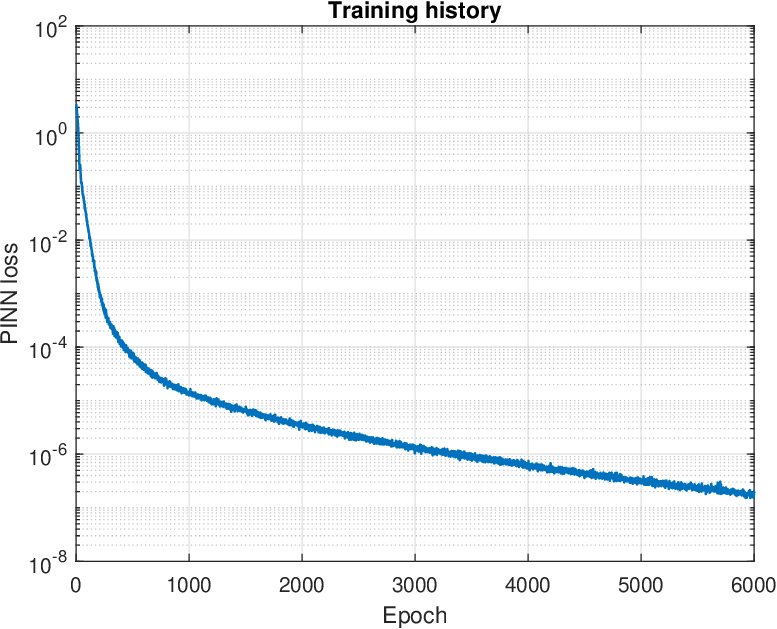}
\end{center}
\caption{Training history for Example 3, $\e_1 = 10^{-3}, \e_2 = 10^{-1}$.}
\label{fig_ex3_train}
\end{figure}

\vspace{1cm}

\noindent
\textbf{Example 4:} 
\vspace{0.2cm}

\noindent
In this final example, we consider a smooth domain in order to show that the method may be extended in a non-tensor-product way to two dimensions. For brevity we do not include the results of the ablation study.

We consider the BVP \eqref{pde_2D}--\eqref{bc_2D}, with $\varepsilon_1 = \varepsilon^2, \varepsilon_2 = 0, c(x,y) = 1$, i.e. a one-parameter reaction-diffusion SPP, where $\Omega \in \mathbb{R}^2$ is now a smooth domain. The load term $f(x,y)$ is chosen so that the exact solution is
\begin{equation*}
	u(x,y) = V(x,y) - V_{\partial \Omega}(\theta)\,B_{\varepsilon}(x,y),
\end{equation*}
where 
\[
V(x,y) = 1 + x + y,
\qquad
V_{\partial \Omega}(\theta) = 1 + R(\theta)\cos\theta + R(\theta)\sin\theta,
\]
and
\begin{equation*}
	B_{\varepsilon}(x,y)
	= \frac{\exp\!\left(-d_r(x,y)/\varepsilon\right) - \exp\!\left(-R(\theta)/\varepsilon\right)}{1 - \exp\!\left(-R(\theta)/\varepsilon\right)}.
\end{equation*}
Here $R(\theta)$, $\theta\in[0,2\pi]$ is the boundary radius function, that describes the distance from the
origin to the boundary at angle $\theta$, and $d_r(x,y)=R(\theta)-r$ is the radial boundary coordinate, i.e., the distance to the boundary measured along the ray of fixed polar angle $\theta$. In general $d_r$ is not the Euclidean signed-distance function. Thus
$d_r(x,y) = R(\theta) - r, \,\, r = \sqrt{x^{2}+y^{2}}$.

We consider a \textit{lima\c{c}on} domain, for which $R(\theta)=1+0.5\cos\theta$. We utilize a configuration of 80 neurons and 2500 epochs, with learning rate $0.01$. We use a Ritz method, i.e.~we minimize
$$
\mathcal{J}(v)=\frac{1}{2} \int_{\Omega}\left(\varepsilon^2|\nabla v|^2+v^2\right) \; d x d y-\int_{\Omega} f v \; d x d y,
$$
with respect to the learnable parameters.
We choose as trial solution
$$
u_{N N}(r, \theta)=S(r \cos \theta, r \sin \theta)-S(R(\theta) \cos \theta, R(\theta) \sin \theta) e^{-(R(\theta)-r) / \varepsilon},
$$
with
$$
S(x, y)=a_0+\sum_{j=1}^n a_j \tanh \left(w^x_{j} x+w^y_{j} y+b_j\right) ,
$$
where $\{a_j, w^x_{j}, w^y_{j}, b_{j}\}$ are learnable parameters. There holds
$$
u_{N N}(R(\theta), \theta)=0,
$$
hence the boundary conditions are strongly enforced.

Table \ref{Ex6_limacon} reports different errors for various values of the perturbation parameter $\varepsilon$. Figure \ref{fig_Ex6_limacon} shows the BL-NN solution and the absolute pointwise error, for $\varepsilon = 10^{-6}$. We mention that each run took about 6 minutes.

\begin{table}[ht]
\caption{Errors over the \textit{lima\c{c}on} using the BL-NN, for Example 4.}
\label{Ex6_limacon}
\centering
\begin{tabular}{ccccc}
\hline
$\varepsilon$ & $L^\infty$ error & $L^2$ error & Relative $L^2$ error & Energy error \\
\hline
$10^{-3}$ & $8.4638 \times 10^{-3}$ & $4.4117 \times 10^{-3}$ & $1.4220 \times 10^{-3}$ & $4.4227 \times 10^{-3}$ \\
$10^{-5}$ & $8.8059 \times 10^{-3}$ & $4.4237\times 10^{-3}$ & $1.4235 \times 10^{-3}$ & $4.4238 \times 10^{-3}$ \\
$10^{-7}$ & $8.8110 \times 10^{-3}$ & $4.4239\times 10^{-3}$ & $1.4235 \times 10^{-3}$ & $4.4239 \times 10^{-3}$ \\
$10^{-9}$ & $8.8115 \times 10^{-3}$ & $4.4239\times 10^{-3}$ & $1.4235 \times 10^{-3}$ & $4.4238 \times 10^{-3}$ \\
$10^{-11}$ & $8.8115 \times 10^{-3}$ & $4.4238 \times 10^{-3}$ & $1.4235 \times 10^{-3}$ & $4.4238 \times 10^{-3}$ \\
\hline
\end{tabular}
\end{table}

\begin{figure}[h!]
	\begin{center}
		\includegraphics[width=0.45\textwidth]{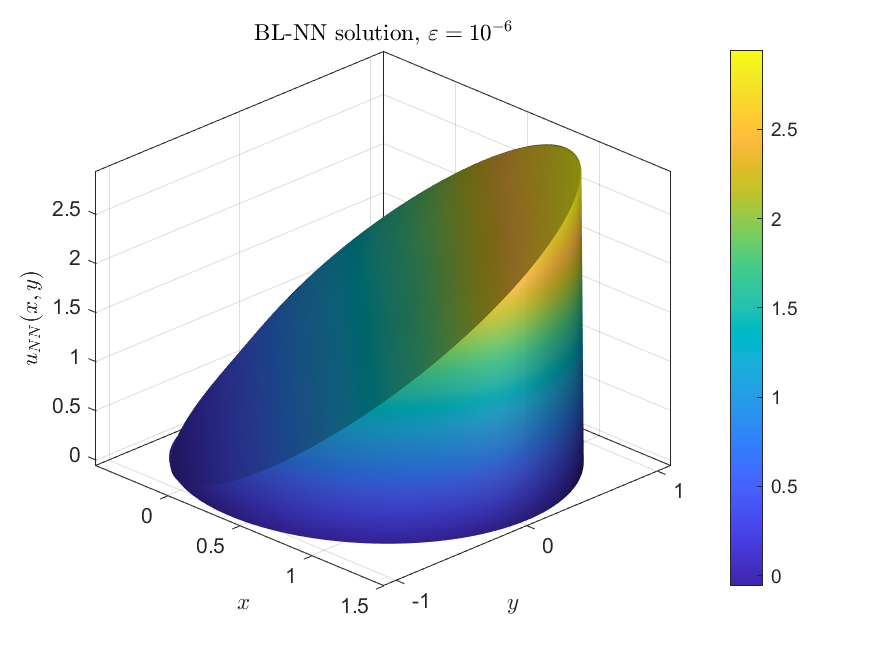}
		\mbox { }
		\includegraphics[width=0.45\textwidth]{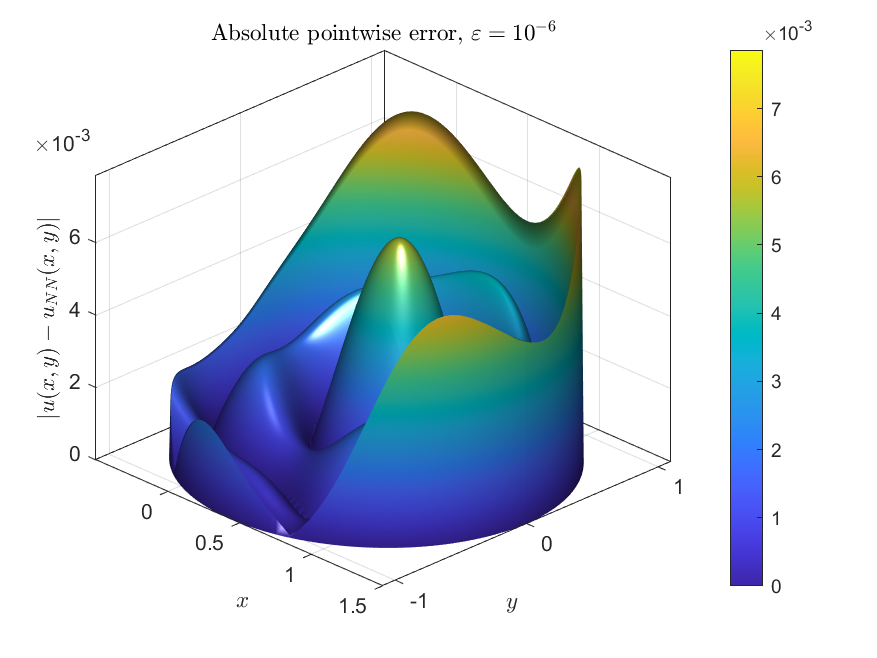}
	\end{center}
	\caption{The BL-NN solution (left), and the absolute pointwise error (right) for $\varepsilon = 10^{-6}$ on the lima\c{c}on.}
	\label{fig_Ex6_limacon}
\end{figure}

For this manufactured reaction-diffusion problem on a smooth domain, the errors remain essentially unchanged over the tested range $10^{-3}\geq\varepsilon\geq10^{-11}$. This experiment therefore provides empirical evidence of parameter-robust behavior for this geometry.



\section{Conclusions}
\label{sec:Concl}
In this article we proposed a novel yet simple shallow $\tanh$ NN enhanced with exponentials, for the emulation of the solution to SPPs. Second order SPPs were considered, and we provided numerical results in one and two dimensions. This was an illustration (and extension) of the results in \cite{OSX} where the existence of a NN with these capabilities was shown.

The numerical experiments indicate that, when reliable information about the dominant layer locations and scales is available and incorporated into the approximation space, the resulting BL-NN can remain accurate over a wide range of perturbation parameters. The principal contribution is therefore an approximation-space design principle: combine a shallow $\tanh$ space for the regular behavior with explicit features reflecting the known asymptotic layer structure, and then couple this space with a suitable training objective. The experiments reported here use residual and energy minimization with Adam, and should be interpreted as evidence of robustness over the tested parameter ranges rather than as a proof of accuracy for all perturbation parameters or all optimization algorithms.

The same principle is potentially applicable to other SPPs for which the dominant layer structure is known, although the appropriate enrichments and boundary treatment must be derived for each problem. Future work includes the application of the proposed NN to more challenging SPPs, such as the \emph{Reissner--Mindlin} plate model (see, e.g., \cite{AD} and the references therein), as well as fluid-flow problems (see, e.g., \cite{TemamSgPert}). A further direction is a systematic study separating the effects of exponential enrichment, layer-adapted sampling, and the stochastic optimization algorithm, including repeated-run statistics.

\vspace{1cm}

\noindent
\textbf{Data Availability}

\noindent
No data was used for this research. The codes are available upon request.

\vspace{0.2cm}
\noindent
\textbf{Conflict of Interest}

\noindent
The authors, C.~Xenophontos and A. Raina, declare that they have no conflict of interest.

\vspace{0.2cm}
\noindent
\textbf{Funding}

\noindent
This research was supported by a University of Cyprus grant, NN4SPP (Dec.~2025 -- Dec.~2027).

\vspace{0.2cm}
\noindent
\textbf{Author Contributions}

\noindent
Both authors, C.~Xenophontos and A. Raina contributed equally to the research.

\vspace{0.2cm}
\noindent


\end{document}